\documentclass{scrartcl}

\usepackage{graphicx,mathrsfs,setspace,xypic,fancyhdr,amsmath,amssymb,amsthm,fullpage,icsinclude,comment}
\usepackage[letterpaper,margin=0.85in]{geometry}
\usepackage[usenames,dvipsnames]{xcolor}
\usepackage[colorlinks, linkcolor=blue, citecolor=blue, urlcolor=red]{hyperref}
\usepackage[all]{xy}

\title{Ulrich Sheaves and Higher-Rank Brill-Noether Theory}

\author{
	Rajesh Kulkarni 
		\footnote{Michigan State University, East Lansing, Michigan. {\sf kulkarni@math.msu.edu}},
	Yusuf Mustopa 
		\footnote{Tufts University, Medford, Massachusetts. {\sf Yusuf.Mustopa@tufts.edu}}, and 
	Ian Shipman 
		\footnote{University of Michigan, Ann Arbor, Michigan. {\sf ishipman@umich.edu}}
}

\theoremstyle{plain}
\newtheorem{thm}{Theorem}[section]
\newtheorem*{thma*}{Theorem A}
\newtheorem*{propb*}{Proposition B}
\newtheorem*{corc*}{Corollary C}
\newtheorem*{propd*}{Proposition D}
\newtheorem*{thm*}{Theorem}
\newtheorem{lma}[thm]{Lemma}
\newtheorem{prop}[thm]{Proposition}
\newtheorem{cor}[thm]{Corollary}

\theoremstyle{definition}
\newtheorem{dfn}[thm]{Definition}

\newtheorem{rmk}[thm]{Remark}
\newtheorem*{rmk*}{Remark}

\def\tk{\tensor_\bk}

\def\bt{\bullet}

\def\rmH{\mathrm{H}}
\def\Proj{\mathrm{Proj}}

\begin{document}

\maketitle

\begin{abstract}
We show that the existence of an Ulrich sheaf on a projective variety $X \subseteq \mathbb{P}^{N}$ is equivalent to the solution of a (possibly) higher-rank Brill-Noether problem for a curve on $X$ that is rarely general in moduli.  In addition, we exhibit a large family of curves for which this Brill-Noether problem admits a solution, and we show that existence of an Ulrich sheaf for a finite morphism of smooth projective varieties of any dimension implies sharp numerical constraints involving the degree of the map and the ramification divisor.
\end{abstract}

\section{Introduction} \label{s:intro}

Let $X \subseteq \mathbb{P}^{n}$ be a projective variety of dimension $m$ and degree $d.$  A coherent sheaf $\cE$ of rank $r$ on $X$ is said to be an \textit{Ulrich sheaf} (with respect to $\cO_{X}(1)$) if for a general linear projection $\pi:X \to \mathbb{P}^{m},$ the direct image sheaf $\pi_{\ast}\cE$ is isomorphic to $\cO_{\mathbb{P}^{m}}^{dr}.$  A straightforward application of Horrocks' Theorem shows that Ulrich sheaves are normalized aCM (arithmetically Cohen-Macaulay) sheaves on $X$ with respect to $\mathcal{O}_{X}(1),$ and that they admit the maximum possible number of global sections for a normalized aCM sheaf of their rank.  Moreover, they are Gieseker-semistable and globally generated.  

The conjecture that every subvariety of projective space admits an Ulrich sheaf has received much attention in recent years, owing to its natural role in the general study of aCM sheaves \cite{CH1} as well as important connections with Chow forms \cite{ESW}, Boij-S\"{o}derberg theory \cite{EE,ES}, an approach to Lech's conjecture \cite{H}, and generalized Clifford algebras \cite{CKM1,CKM3}.  There are affirmative answers for curves \cite{ESW}, hypersurfaces \cite{BHS}, complete intersections and linear determinantal varieties \cite{BHU}, Grassmannians \cite{CM}, Segre varieties \cite{ESW,CMP}, and generic K3 surfaces of any genus \cite{AFO}.  Our main result (Theorem A) provides, among other things, a geometric gauge (Corollary C) for the difficulty of this conjecture:  an affirmative answer effectively solves a series of higher-rank Brill-Noether problems on a wide variety of non-generic curves all at once. 

Before going further, we emphasize a fundamental difference between the point of view taken here and previous constructions involving the geometry of curves.  In the papers \cite{AFO,CKM1,CKM2} Ulrich sheaves of high rank on surfaces are produced from special line bundles on a curve via elementary modification, and in \cite{CH2} Ulrich sheaves on smooth cubic threefolds are produced from curves using a Serre-type construction.  The starting point of our approach is more na\"{i}ve:  given a projective variety $X \subset \mathbb{P}^{N}$ of dimension 2 or greater, does there exist a curve $C \subset X$ and a vector bundle on $C$ which is the restriction of an Ulrich sheaf on $X$?   

Turning to details, observe that the definition of Ulrich sheaf depends only on a direct image under a morphism.  Given a degree-$d$ finite flat morphism $f:X \rightarrow Y$ of projective varieties, we define a coherent sheaf $\cE$ of rank $r$ on $X$ to be $f-$Ulrich if $f_{\ast}\cE \cong \cO_{Y}^{dr}.$  Note that pullback by an embedding $Y' \hookrightarrow Y$ takes an $f-$Ulrich sheaf on $X$ to an $f'-$Ulrich sheaf on $X':=X \times_{Y} Y',$ where $f':X' \to Y'$ is the induced morphism.  This abstract framework is a proper setting in which to consider ``ordinary" Ulrich sheaves, as indicated by the following result. 

%Elementary considerations of the Riemann theta-divisor imply that every smooth curve in $\mathbb{P}^{n}$ admits an Ulrich sheaf of rank 1.  While the surface and threefold cases of the existence problem are still largely open, they translate to problems in the geometry of curves via elementary modification \cite{AFO,CKM1,CKM2} and the Serre construction \cite{CH2}, respectively.  Using a rather different approach, our main result shows that the existence of Ulrich sheaves on a projective variety of \textit{any} dimension is a problem in the geometry of curves. 

\begin{thma*}
\label{main}
Let $X \subseteq \mathbb{P}^{n}$ be a projective variety of dimension $m \geq 2$ and degree $d$ whose homogeneous ideal $I_{X|\mathbb{P}^{n}}$ is generated in degree at most $N$ (where $N \geq 2$), and let $f:X \rightarrow \mathbb{P}^{m}$ be a general linear projection.  In addition, let $C \subseteq \mathbb{P}^{m}$ be an aCM curve whose homogeneous ideal $I_{C|\mathbb{P}^{m}}$ is generated in degrees at least $N+1,$ satisfying the property that $X_{C}:=X \times_{\mathbb{P}^{m}} C$ is a smooth irreducible curve.  Then any $f|_{C}-$Ulrich sheaf on $X_{C}$ extends to an Ulrich sheaf for $X.$ 
\end{thma*}

The main ingredient of the proof is Proposition \ref{prop:general-lift}, a general result on the lifting of sheaves via finite flat coverings.  We say that a coherent sheaf $\cF$ on $Y$ \textit{lifts to} $X$ via $f:X \to Y$ if $\cF \cong f_{\ast}\cE$ for some coherent sheaf $\cE$ on $X$; with this terminology, the existence problem for $f-$Ulrich sheaves is the question of whether some trivial sheaf on $Y$ lifts to $X$ via $f.$  More specifically, Proposition \ref{prop:general-lift} gives a criterion for lifting \textit{dissoci\'{e}} sheaves on $Y,$ i.e. direct sums of twists of a fixed ample line bundle on $Y.$  In the special case $Y=\mathbb{P}^{n},$ the lifts to $X$ of dissoci\'{e} sheaves via $f$ are precisely the sheaves on $X$ that are aCM with respect to $\cO_{X}(1).$  Consequently, we expect that our work can be applied to the construction of non-Ulrich aCM sheaves; see Lemma \ref{lma:acm-check}.               

Theorem A justifies a thorough analysis of the existence problem for $f-$Ulrich sheaves when $f$ is a finite morphism of smooth projective curves.  When $f:X \to \mathbb{P}^{1}$ is a degree-$d$ branched covering of $\mathbb{P}^{1}$ by a smooth projective curve of genus $g,$ a vector bundle $\cE$ of rank $r$ on $X$ is Ulrich precisely when $c_{1}(\cE)=r(g-1+d)$ and $h^{0}(\cE(-1))=h^{0}(\cE(-1))=0.$  The following proposition gives an equally clean characterization for the general case.       

\begin{propb*}
\label{ulrich-curves-criterion}
Let $f:X \to Y$ be a degree-$d$ finite morphism of smooth projective curves, and let $b$ be the degree of the branch divisor of $f.$  Then for a vector bundle $\cE$ of rank $r$ on $X,$ the following are equivalent.
\begin{itemize}
\item[(i)]{$\cE$ is $f-$Ulrich.}
\item[(ii)]{$\cE$ is globally generated, $h^{0}(\cE)=dr,$ and $c_{1}(\cE)=br/2.$}   
\end{itemize}  
\end{propb*}

Just as Ulrich sheaves for a subvariety of projective space are Gieseker-semistable with respect to the hyperplane polarization, $f-$Ulrich sheaves are Gieseker-semistable with respect to the pullback via $f$ of any given polarization on $Y$ (the proofs are virtually identical).  Proposition B then implies that the S-equivalence classes of $f-$Ulrich sheaves of rank $r$ on the curve $X$ are parametrized by a Brill-Noether locus in the moduli space $U_{X}(r,br/2)$ of semistable vector bundles on $X$ having rank $r$ and degree $br/2$; we may then rephrase the existence problem for Ulrich sheaves on projective varieties as a higher-rank Brill-Noether problem for curves.   It must be emphasized that the curves in question are special in moduli; indeed, they are coverings of aCM curves with homogeneous ideals generated in degree 3 or greater, all of which are irrational. 
          
\begin{corc*}
\label{cor-main}
Assume the hypotheses of Theorem A.  Then $X \subseteq \mathbb{P}^{n}$ admits an Ulrich sheaf of rank $r$ if and only if $X_{C}$ admits a globally generated vector bundle of rank $r$ with degree $br/2$ and $dr$ global sections.
\end{corc*}

At this point it is natural to ask for examples of $f-$Ulrich bundles which are not obvious restrictions of ordinary Ulrich bundles.   A family of such examples is provided by correspondences of trivial type on products of curves, via Proposition \ref{prop:prod-curves}.  In the simplest case, $X$ is a member of the linear system $|\cL \boxtimes \cM|$ in $C \times Z$ for globally generated line bundles $\cL$ and $\cM$ of positive degree on $C$ and $Z,$ respectively; if $\cE$ is any vector bundle on $Z$ which is Ulrich with respect to $\cM$, and $f : X \to C$ is induced by projection, then the restriction of $\cO_{C} \boxtimes \cE$ to $X$ is $f-$Ulrich.  Here $\cE$ can have rank 1, resulting in an $f-$Ulrich line bundle. Ulrich bundles of rank 1 are rare enough that we do not expect such a case to ``occur in nature."  However, Proposition \ref{prop:prod-curves} also produces morphisms admitting higher-rank $f-$Ulrich bundles that cannot obviously be replaced by line bundles (compare Lemma \ref{lma:gf-ulrich}). 

Using a Clifford-type theorem for semistable vector bundles on curves due to V. Mercat (Theorem 2.1 in \cite{Mer}), we obtain a necessary condition for the existence of an $f-$Ulrich bundle, namely that $f$ has ``enough" ramification.  This generalizes with little effort to finite morphisms of smooth varieties of higher dimension (Corollary \ref{cor-bound-final-form}).

%We can say a little more when ${\rm Pic}(X)$ is discrete, i.e. when $H^{1}(\cO_{X})=0.$  (This includes all varieties of dimension at least 2 on which Ulrich sheaves are currently known to exist.)  The determinant of any relevant bundle on $X_C$ must be the restriction of     

\begin{propd*}
Let $f:X \to Y$ be a degree-$d$ finite morphism of smooth projective curves whose branch divisor has degree $b$, where $d \geq 3,$ the Clifford index of $X$ is 2 or greater (i.e. $X$ is not hyperelliptic, trigonal, or a plane quintic), and $Y$ has positive genus.  Then if an $f-$Ulrich sheaf exists, we have $b \geq 4d,$ and this inequality is sharp for all $d \geq 3.$   
\end{propd*} 

The boundary example we construct for each $d \geq 3$ is a degree-$d$ cover $f$ of a hyperelliptic curve $Y$ by a bielliptic curve $X;$ we show that $f-$Ulrich line bundles exist, and also that they are parametrized by the elliptic curve of which $X$ is a double cover (Proposition \ref{prop:bielliptic-boundary}).  It must be noted that Theorem A cannot be directly applied to these examples, since the homogeneous ideal of an aCM hyperelliptic curve contains generators of degree 2 (compare Remark \ref{rmk:high-degree}).

\subsection*{Acknowledgments}

We would like to thank Ragnar-Olaf Buchweitz, Hailong Dao, and Rob Lazarsfeld for helpful discussions. Ian Shipman and Rajesh Kulkarni were partially supported by National Science Foundation under Grant No.~0932078-000 while in residence at the Mathematical Sciences Research Institute during the spring 2013 semester on noncommutative algebraic geometry.  Rajesh Kulkarni was partially supported by the National Science Foundation awards DMS-1004306 and  DMS-1305377. Ian Shipman was partially supported by the National Science Foundation award DMS-1204733.

\section{Preliminaries} \label{s:Prelims}

We work over an algebraically closed field $\bk$ of characteristic 0.  For the rest of the article, $S/R$ denotes a finite, flat and generically unramified degree-$d$ graded extension $R \subset S$ of finitely generated commutative graded $\bk-$algebras $R$ and $S$ which are both generated in degree 1 with degree 0 part equal to $\bk.$

\begin{dfn}
Given a finitely generated graded $R-$module $N,$ a finitely generated graded $S-$module $M$ is a \textit{lift of $N$ via $S/R$} if $M$ is isomorphic to $N$ as a graded $R-$module. 
\end{dfn}

In the sequel we will use that fact that a lift of $N$ for $S/R$ is equivalent to the datum of a graded $R-$linear map $S \otimes_{R} N \to N$ which gives $N$ the structure of a graded $S-$module. 

Our main concern is the case where $N$ is a free graded $R-$module, i.e. when $N \cong \oplus_{j=1}^{t}R(-a_{j})$ for integers $a_{1}, \cdots ,a_{j}.$  The following result is elementary and straightforward, but we give an argument below due to its importance for the proof of Theorem A.    

\begin{prop}
\label{prop:normal-lift}
Assume $R$ is normal, and let $N \cong \oplus_{j=1}^{t}R(-a_{j}).$  If $X={\rm Proj}(S),Y={\rm Proj}(R),$ and $f:X \to Y$ is the map induced by $S/R,$ then the standard formation of graded $S-$modules from coherent sheaves on $X$ yields a one-to-one correspondence between lifts of $N$ for $S/R$ and coherent sheaves on $X$ whose direct image via $f$ is $\oplus_{j=1}^{t}\cO_{Y}(-a_{j}).$ 
\end{prop}

\begin{proof}
Let $\cE$ be a coherent sheaf on $X$ for which $f_{\ast}\cE \cong \oplus_{j=1}^{t}\cO_{Y}(-a_{j}).$  Then $M:=\oplus_{i \in \mathbb{Z}}\rmH^{0}(\cE(i))$ is a finitely generated graded $S-$module, and we have that as a graded $R-$module,
\begin{equation}
M=\oplus_{i \in \mathbb{Z}}\rmH^{0}(\cE(i)) \cong \oplus_{i \in \mathbb{Z}}\rmH^{0}(f_{\ast}\cE(i)) \cong \oplus_{j=1}^{t}\oplus_{i \in \mathbb{Z}}\rmH^{0}(\cO_{Y}(-a_{j}+i))
\end{equation}
By the assumption that $R$ is normal and generated in degree 1, we have that $R(a_{j}) \cong \oplus_{i \in \mathbb{Z}}\rmH^{0}(\cO_{Y}(-a_{j}+i))$ for $1 \leq j \leq t,$ so we may conclude that $M$ is isomorphic to $N$ as a graded $R-$module, i.e. that $M$ is a lift of $N$ for $S/R.$  

Conversely, starting with a lift of $N$ for $S/R,$ it is even easier to see that the associated coherent sheaf $\cE$ on $X$ satisfies $f_{\ast}\cE \cong \oplus_{j=1}^{t}\cO_{Y}(-a_{j}),$ and that the two constructions are inverse to one another.
\end{proof}

\begin{cor}
Under the hypotheses of Proposition \ref{prop:normal-lift}, a coherent sheaf on $X$ is $f-$Ulrich if and only if for some positive integer $t,$ its induced graded $S-$module is a lift of $R^{\oplus t}$ for $S/R.$ \hfill \qedsymbol
\end{cor}

We will also need the following lemma.   

\begin{lma} \label{lem:restriction}
Let $J \subset R$ be a homogeneous prime ideal such that $Z = \Proj(R/J) \subset Y = \Proj(R)$ is nonempty, and $X_Z = \Proj(S/J S)$.  Write $g:X_Z \to Z$ for the morphism induced by $R/J \to S/JS$.  Suppose $N \cong \oplus_{j=1}^t R(a_j)$ and let $\wt{\alpha}:S \tensor_R N \to N$ be a graded $R$-module homomorphism.  Let $\cE$ be the sheafification of $N$.  Then the induced morphism $\alpha:S/JS \tensor_{R/J} N/JN \to N/JN$ equips $N/J N$ with an $S/JS$-module structure if and only if the morphism $g_* \cO_{X_Z} \tensor_{\cO_Z} \cE|_Z \to \cE|_Z$ equips $\cE|_Z$ with a $g_*\cO_{X_Z}$-module structure.
\end{lma}
\begin{proof}
If $\mu:S/J S \tensor_{R/J} S/J S \to S/J S$ is the multiplication map and $i:N/J N = R/J \tensor_{R/J} N/JN \to S/J \tensor_{R/J} N/J N$ is the natural inclusion, then $\alpha: S/JS \tensor_{R/J} N/JN \to N/JN$ equips $N/JN$ with an $S/JS-$module structure if and only if (i) $\alpha(1 \tensor n ) = n$ for all $n \in N/JN$ and (ii) $\alpha \circ (\mu \tensor \id) = \alpha \circ (\id \tensor \alpha)$ as maps.  We can then reformulate (i) and (ii) as the vanishing of the $R/J$-module maps $i - \id$ and $\alpha \circ (\mu \tensor \id) - \alpha \circ (\id \tensor \alpha)$.  In general, if $M \to M'$ is a map of finitely generated, graded $R$-modules which induces the zero map on sheafifications, the image of the map is contained in the submodule of elements annihilated by a power of the irrelevant ideal of $R$.  Since $\Proj(R/J)$ is nonempty, $N/JN$ has no elements that are annihilated by a power of the irrelevant ideal of $R$.  This means that $i - \id$ and $\alpha \circ (\mu \tensor \id) - \alpha \circ (\id \tensor \alpha)$ are zero if and only if they are zero after sheafification.  This concludes the proof.
\end{proof}

\section{Reduction to curves} \label{s:extension}

This section contains the proof of Theorem A.  For the rest of the section, we fix a finite flat morphism $f:X \to Y$ of projective varieties and a very ample line bundle $\cO_{Y}(1)$ on $Y$ such that $\cO_{X}(1):=f^{\ast}\cO_{Y}(1)$ is very ample, and we let $S=\oplus_{i \geq 0}\rmH^{0}(X,\mathcal{O}_{X}(i))$, $R:=\oplus_{i \geq 0}\rmH^{0}(Y,\cO_{Y}(i))$.  We define a subscheme $Z \subset Y$ to be \emph{$m$-separating} if for $0 \leq i \leq m$ we have $\rm{H}^0(Y,\cI_{Z|Y}(i)) = 0.$
 
\begin{prop}
\label{prop:general-lift}
Assume $R$ is normal.  Let $\cE = \oplus_{j=1}^{t}\cO_{Y}(-a_{j})$ for some $a_1, \cdots ,a_{j} \in \mathbb{Z}.$  Then there exists a positive integer $\underline{n} = \underline{n}(f,{\cO_Y(1)},\cE)$ such that if $Z \subset Y$ is an $\underline{n}-$separating subscheme satisfying: 
\begin{itemize}
\item[(i)]{$\cE|_Z$ lifts to $X_Z := X \times_Y Z$.}
\item[(ii)]{The associated morphism $(f_Z)_*\cO_{X_Z} \tensor_{\cO_Z} \cE|_Z \to \cE|_Z$ endowing $\cE|_{Z}$ with the structure of an $f_{\ast}\cO_{X_Z}-$module is the restriction to $Z$ of an $\cO_Y$-linear morphism $f_*\cO_X \tensor_{\cO_Y} \cE \to \cE.$}
\end{itemize} 
then $\cE$ lifts to $X$.
\end{prop}

\begin{proof}
Without loss of generality, we may assume $\cE$ to be normalized, so that $a_{t} \geq \cdots \geq a_{2} \geq a_{1} \geq 0$ and at least one of the $a_{j}$ is zero.  Let $M=\oplus_{i \geq 0}\rm{H}^{0}(Y,\cE(i)).$  Consider the natural map of graded $\bk$-algebras  
\begin{equation}
\wt{R} := R \otimes_\bk \Sym^{\bt}\rm{H}^0(X,\cO_X(1)) \to S
\end{equation}
let $I_X$ be the kernel, and put $\wt{S} = \wt{R} / I_X$.  Since $R$ is normal and $\Sym^{\bt}\rm{H}^0(X,\cO_{X}(1))$ is generated in degree 1, it follows at once that $\wt{R}$ and $\wt{S}$ are generated in degree 1.  The morphism of projective varieties induced by the map $\wt{R} \rightarrow S$ may then be described as 
\begin{equation}
X \hookrightarrow {\cJ}(Y,\mathbb{P}H^{0}(\mathcal{O}_{X}(1))) \hookrightarrow \mathbb{P}(H^{0}(\mathcal{O}_{Y}(1)) \oplus H^{0}(\mathcal{O}_{X}(1)))
\end{equation}
where ${\cJ}(Y,\mathbb{P}H^{0}(\mathcal{O}_{X}(1)))$ is the linear join of $Y \subseteq \mathbb{P}H^{0}(\mathcal{O}_{Y}(1))$ and $\mathbb{P}H^{0}(\mathcal{O}_{X}(1)).$  We define $n_X$ to be the maximum degree of an element in a minimal generating set of $I_X$ (note that $\wt{R}$ is graded by total degree).  It will be shown that $\underline{n} := \max\{2,n_X\} + a_t$ satisfies the desired property.  

Let $Z \subset Y$ be an $\underline{n}$-separating subscheme satisfying conditions (i) and (ii), let $J \subseteq R$ be the homogeneous ideal of $Z,$ and let $\alpha:f_*\cO_X \tensor_{\cO_Y} \cE \to \cE$ be the $\cO_{Y}$-linear map of sheaves whose restriction $\alpha|_Z$ equips $\cE|_Z$ with an $(f_Z)_*\cO_{X_Z}$-module structure.  By Proposition \ref{prop:normal-lift} and Lemma \ref{lem:restriction}, this corresponds to a map ${\alpha}' : S \tensor_{R} M \to M$ of graded $R-$modules for which the induced map $\bar{\alpha}': S/JS \tensor_{R/J} M/JM \to M/JM$ equips $M/JM$ with a graded $S/JS$-module structure.  Note that by the $\underline{n}-$separatedness of $Z$ we have $J_i = 0$ if $i \leq \underline{n},$ so that the maps $S \to S/JS$ and $M \to M/JM$ are isomorphisms in degrees less than or equal to $\underline{n}$; an immediate consequence (which will be used momentarily) is that ${\alpha}'$ and $\bar{\alpha}'$ can be identified for elements of $S \otimes_{R} M$ having sufficiently low degree.  We will be done once we check that ${\alpha}'$ defines a graded $S$-module structure on $M$.  

Observe that ${\alpha}'$ induces a graded $T^{\bt}\rm{H}^{0}(\mathcal{O}_{X}(1))-$module structure on $M$ as follows:  for $x_{1} , \cdots ,x_{k} \in \rm{H}^{0}(\mathcal{O}_{X}(1))$ and $m \in M,$
\begin{equation}
(x_{1} \tensor \cdots \tensor x_{k}) \cdot m := {\alpha}'(x_{1} \tensor ({\alpha}'( \cdots \otimes {\alpha}'(x_{k} \otimes m))))
\end{equation}
This in turn determines a natural $R \otimes_{\bk} T^{\bt}\rm{H}^{0}(\mathcal{O}_{X}(1))$-module structure on $M$.  Suppose that $x,y \in \rm{H}^0(X,\cO_X(1))$ and $m \in M_i$, with $i \leq a_t$.  Then since $x \tensor {\alpha}'(y \tensor m)$ has degree at most $2+a_{t}$, we see from our assumption on $\bar{\alpha}'$ that  
\begin{equation}
{\alpha}'(x \tensor {\alpha}'(y \tensor m)) = \bar{\alpha}'(x \tensor \bar{\alpha}'(y \tensor m)) = \bar{\alpha}'(y \tensor \bar{\alpha}'(x \tensor m))
\end{equation}  
It then follows from symmetry that ${\alpha}'(x \tensor {\alpha}'(y \tensor m)) = {\alpha}'( y \tensor {\alpha}'(x \tensor m))$.  As a consequence, the $R \otimes_{\bk} T^{\bt}\rm{H}^{0}(\mathcal{O}_{X}(1))$-module structure on $M$ induces a map $\wt{\alpha}' : \Rt \tensor_{R} M \to M$ which gives $M$ a graded $\Rt$-module structure.  Now, for $r \in I_X$ of degree at most $n_X$ and $m \in M$ of degree at most $a_{t},$ the fact that $r$ is zero in $S/JS$ implies that ${\alpha}'(r \tensor m) = \bar{\alpha}'(r \tensor m) = 0$.  Since $M$ is generated in degrees up to $a_{t}$ and $I_X$ is generated in degrees up to $n_X$ we see that $\wt{\alpha}'$ annihilates $I_{X} \otimes_{R} M$, and therefore descends to ${\alpha}' : S \tensor_{R} M \to M$.  This concludes the proof.
\end{proof}

We will now see that the condition (ii) in the statement of Proposition \ref{prop:general-lift} is satisfied in many cases of interest. 

\begin{lma}
\label{lma:acm-check}
Let $f:X \to \mathbb{P}^{m}$ be a finite flat morphism, and let $C \subset \mathbb{P}^{m}$ be an $\underline{n}-$separating aCM curve in $\mathbb{P}^{m}$.  Then if $\cE \cong \oplus_{j=1}^{dr}\cO_{\mathbb{P}^{m}}(-a_j)$ where $\displaystyle\max_{i,j}|a_{i}-a_{j}| < \underline{n},$ the curve $C$ satisfies condition (ii) in Proposition \ref{prop:general-lift}.     
\end{lma}

\begin{proof}
The obstruction to extending a $\cO_C$-linear map $f_{\ast}\cO_{X_C} \otimes \cE|_{C} \to \cE|_{C}$ to an $\cO_{\mathbb{P}^{m}}$-linear map $f_{\ast}\cO_{\mathbb{P}^{n}} \otimes \cE \to \cE$ lies in $\rmH^{1}(f_{\ast}\cO_{X} \otimes {\cE}nd(\cE) \otimes \cI_{C|\mathbb{P}^{m}}).$  The aCM condition implies that by the Auslander-Buchsbaum formula, $\cI_{C|\mathbb{P}^{m}}$ has a minimal free resolution of the form
\begin{equation}
\label{eqn:ideal-res}
0 \leftarrow \cI_{C|\mathbb{P}^{m}} \leftarrow \bigoplus_{j_{1}} \cO_{\mathbb{P}^{m}}(-b_{1,j_1}) \leftarrow \cdots \leftarrow \bigoplus_{j_{m-1}}\cO_{\mathbb{P}^{m}}(-b_{m-1,j_{m-1}}) \leftarrow 0
\end{equation}
where for all $\ell,s$ we have $b_{\ell,j_s} \geq \underline{n}.$  Since ${\cE}nd(\cE) \cong \oplus_{i,j}\cO_{\mathbb{P}^{m}}(a_{j}-a_{i}),$ if we chop (\ref{eqn:ideal-res}) into short exact sequences, twist all of them by $f_{\ast}\cO_{X} \otimes {\cE}nd(\cE),$ and take cohomology, we obtain an inclusion
\begin{equation}
\rmH^{1}(f_{\ast}\cO_{X} \otimes {\cE}nd(\cE) \otimes \cI_{C|\mathbb{P}^{m}}) \hookrightarrow \bigoplus_{i,j,j_{m-1}}\rmH^{m-1}(f_{\ast}\mathcal{O}_{X} \otimes \cO_{\mathbb{P}^{m}}(a_{j}-a_{i}-b_{m-1,j_{m-1}})) %\cong \oplus_{i,j,j_{m-1}}H^{m-1}(\cO_{X}(a_{i}-a_{j}-b_{m-1,j_{m-1}})
\end{equation} 
Since each summand on the right-hand side is isomorphic to $\rmH^{m-1}(\cO_{X}(a_{j}-a_{i}-b_{m-1,j_{m-1}}))$ which is zero by Kodaira vanishing, we conclude that $\rmH^{1}(f_{\ast}\cO_{X} \otimes {\cE}nd(\cE) \otimes \cI_{C|\mathbb{P}^{m}})=0.$
\end{proof}
 
\textit{Proof of Theorem A:}  Since the Ulrich condition depends only on generic linear projections, we may assume without loss of generality that $X$ is embedded in $\mathbb{P}^{n}$ by the complete linear series $|\mathcal{O}_{X}(1)|.$  By Lemma \ref{lma:acm-check} it suffices to check that if $I_{X|\mathbb{P}^{n}}$ is generated in degree at most $N$ we may take $\underline{n}=N$ in the proof of Proposition \ref{prop:general-lift}.  In the case at hand, $Y=\mathbb{P}(V)$ where $V \subseteq H^{0}(\mathcal{O}_{X}(1))$ is a general $(m+1)-$dimensional subspace, and $f:X \rightarrow \mathbb{P}(V)$ is a general linear projection associated to $X \hookrightarrow \mathbb{P}H^{0}(\cO_X(1)).$  Using the notation of the proof of Proposition \ref{prop:general-lift}, all we need to do is verify that the ideal $I_{X} \subseteq \wt{R}$ is generated in degree at most $N.$  We have 
\begin{equation}
\wt{R}={\rm Sym}^{\bt}V \otimes_{\bk} {\rm Sym}^{\bt}H^{0}(\cO_{X}(1)) \cong {\rm Sym}^{\bt}(V \oplus H^{0}(\cO_{X}(1)))
\end{equation}
where $V$ is an $(m+1)-$dimensional subspace of $H^{0}(\cO_{X}(1)).$  It follows from the exact sequence
\begin{equation}
V \otimes_{\bk}  {\rm Sym}^{\bt}(V \oplus H^{0}(\cO_{X}(1))) \rightarrow  {\rm Sym}^{\bt}(V \oplus H^{0}(\cO_{X}(1))) \rightarrow  {\rm Sym}^{\bt}H^{0}(\cO_{X}(1)) \rightarrow 0
\end{equation}
that $I_{X} \subseteq \wt{R}$ is generated by $I_{X|\mathbb{P}^{n}}$ and $V$; in particular, $I_{X}$ is generated in degree at most $N$. \hfill \qedsymbol
        
\section{The Case of Curves} \label{s:curves}

In this final section we prove Propositions B and D, and study some explicit examples of $f-$Ulrich sheaves.

\textit{Proof of Proposition B:}  Let $f:X \rightarrow Y$ be a degree-$d$ finite morphism of smooth projective curves with branching degree equal to $b,$ and let $E$ be a vector bundle of rank $r$ on $X.$  In what follows, we denote the respective genera of $X$ and $Y$ by $g_{X}$ and $g_{Y}.$

((i) $\Rightarrow$ (ii)) If $E$ is $f-$Ulrich, we have by definition an isomorphism $\mathcal{O}_{Y}^{dr} \xrightarrow[]{\cong} f_{\ast}E$; the adjunction of this is $\mathcal{O}_{X}^{dr} \xrightarrow[]{\cong} f^{\ast}f_{\ast}E \rightarrow E,$ which is surjective.  It follows that $E$ is globally generated with $dr$ global sections.  By Riemann-Roch and Riemann-Hurwitz, we have 
\begin{equation}
\label{chern-branch}
c_{1}(E)=\chi(E)-r(1-g_{X})=\chi(f_{\ast}E)-r(1-g_{X})=dr(1-g_{Y})-r(1-g_{X})=br/2.
\end{equation}

((ii) $\Rightarrow$ (i))  If $E$ is a vector bundle of rank $r$ which is globally generated with $dr$ global sections and $c_{1}(E)=br/2$, we have the evaluation sequence
\begin{equation}
0 \rightarrow M \rightarrow \mathcal{O}_{X}^{dr} \rightarrow E \rightarrow 0
\end{equation}
and its direct image
\begin{equation}
0 \rightarrow f_{\ast}M \rightarrow (f_{\ast}\mathcal{O}_{X})^{dr} \rightarrow f_{\ast}E \rightarrow 0
\end{equation}
If the tracial summand $\mathcal{O}_{Y}^{dr} \subset (f_{\ast}\mathcal{O}_{X})^{dr}$ has nonzero intersection with the subbundle $f_{\ast}M,$ then since $h^{0}(f_{\ast}M)=h^{0}(M)=0$ there is a section $s: \mathcal{O}_{Y} \rightarrow \mathcal{O}_{Y}^{dr}$ and an effective divisor $Z$ on $Y$ such that $f_{\ast}M \cap s(\mathcal{O}_{Y}) = \mathcal{O}_{Y}(-Z).$  However, this would imply that the torsion sheaf $\mathcal{O}_Z$ is a subsheaf of the vector bundle $f_{\ast}E,$ which is impossible.  It follows that the map $\mathcal{O}_{Y}^{dr} \subset (f_{\ast}\mathcal{O}_{X})^{dr} \rightarrow f_{\ast}E$ is injective.  The support of its cokernel is a divisor in $|\det{f_{\ast}E}|,$ but a calculation similar to (\ref{chern-branch}) shows that $c_{1}(f_{\ast}E)=0,$ so the map is also surjective.  We may then conclude that $\mathcal{O}_{Y}^{dr} \cong f_{\ast}E,$ i.e. that $E$ is $f-$Ulrich.  \hfill \qedsymbol

\subsection{Examples of $f-$Ulrich Sheaves}

In what follows, $S$ is a smooth surface, $f : S \to C$ is a morphism, and $\cG$ is a vector bundle on $S.$

%\begin{dfn}
%Let $S$ be a smooth surface, let $f : S \to C$ be a morphism, and let $L$ be a line bundle on $S.$  We say that a coherent sheaf $\cF$ is $(L,f)-$Ulrich if $f_{\ast}(\cF \otimes L^{\vee}) = R^{1}f_{\ast}(\cF \otimes L^{\vee})=0$ and $f_{\ast}\cF$ is trivial.
%\end{dfn}

\begin{dfn}
A coherent sheaf $\cF$ is $(\cG,f)-$Ulrich if $R^{i}f_{\ast}(\cF \otimes {\cG}^{\vee})=0$ for $i=0,1$ and $f_{\ast}\cF$ is a trivial vector bundle on $C.$
\end{dfn}

\begin{lma}
\label{lma:gf-ulrich}
If $\cF$ is a locally free $(\cG,f)-$Ulrich sheaf of rank $r$ on $S$ and $X \subset S$ is a smooth irreducible curve which is the degeneracy locus of a vector bundle morphism $\phi: \cO_{S}^{\oplus {\rm rk}(\cM) \cdot {\rm rk}(\cG)} \to f^{\ast}\cM \otimes \cG$ for some vector bundle $\cM$ on $C,$ then $\cF|_{X} \otimes {\rm coker}({\phi}^{\vee})$ is an $f|_{X}$-Ulrich sheaf.
\end{lma}

\begin{proof}
Consider the exact sequence
\begin{equation}
0 \to \cF \otimes f^{\ast}\cM^{\vee} \otimes {\cG}^{\vee} \to \cF^{\oplus {\rm rk}(\cM) \cdot {\rm rk}(\cG)} \to {\cF}|_{X} \otimes {\rm coker}({\phi}^{\vee}) \to 0 
\end{equation}
Since $\cF$ is $(\cG,f)-$Ulrich by assumption, applying $f_{\ast}$ to this exact sequence gives isomorphisms 
\begin{equation}
\cO_{C}^{\oplus {\rm rk}(\cM) \cdot {\rm rk}(\cG) \cdot {\rm rk}(f_{\ast}\cF)} \cong (f_{\ast}\cF)^{\oplus {\rm rk}(\cM) \cdot {\rm rk}(\cG)} \cong (f|_{X})_{\ast}(\cF|_{X} \otimes {\rm coker}(\phi^{\vee}))
\end{equation}
\end{proof}
%\begin{prop}
%\label{prop:lf-ulrich}
%If $\cF$ is a locally free $(L,f)-$Ulrich sheaf on $S$ and $X \subset S$ is a curve defined by a global section of $|f^{\ast}\cM \otimes L|$ for some line bundle $\cM$ on $C,$ then $\cF|_{X}$ is $f|_{X}-$Ulrich.
%\end{prop}
%\begin{proof}
%Consider the exact sequence
%\begin{equation}
%0 \to \cF \otimes f^{\ast}\cM^{\vee} \otimes L^{\vee} \to \cF \to \cF|_{X} \to 0 
%\end{equation}
%Since $\cF$ is $(L,f)-$Ulrich by assumption, applying $f_{\ast}$ to this exact sequence gives an isomorphism $\cO_{C}^{\oplus {\rm rk}(f_{\ast}\cF)} \cong f_{\ast}\cF \cong (f|_{X})_{\ast}(\cF|_{X}).$
%\end{proof}

We now apply this to the case where $S$ is a product of curves. 

\begin{prop}
\label{prop:prod-curves}
Let $C$ and $Z$ be smooth projective curves with respective genera $g$ and $h.$  Assume the following:
\begin{itemize}
\item[(i)]{$Z$ admits a pair $\cF$ and $\cG$ of vector bundles such that $\cG$ is globally generated of positive degree and $H^{i}(\cF \otimes \cG^{\vee})=0$ for $i=0,1.$}
\item[(ii)]{There is a vector bundle $\cM$ on $C$ and a morphism $\phi : \cO_{C \times Z}^{\oplus {\rm rk}(\cM) \cdot {\rm rk}(\cG)} \to \cM \boxtimes \cG$ of vector bundles on $C \times Z$ whose degeneracy locus $X$ is a smooth irreducible curve.}
\end{itemize}   
Then the finite map $f: X \to C$ induced by projection admits an $f-$Ulrich sheaf.  Moreover, for each $r \geq 1,$ there exist bundles $\cF$ and $\cG$ on $Z$ satisfying (i) and a line bundle $\cM$ on $C$ satisfying (ii) for which the resulting $f-$Ulrich bundle is of rank at least $r.$  
%Let $f: X \to C$ be a degree-$d$ morphism of smooth projective curves, and assume there exists a finite morphism $g: X \to Z$ of degree $e$ satisfying the following properties:
%\begin{itemize}
%\item[(i)]{$Z$ admits a globally generated semistable vector bundle $\cG$ of rank $t \geq 1$ and degree $dk \geq d$ which admits a theta-divisor in the appropriate moduli space.}
%\item[(ii)]{There exists $\cM \in {\rm Pic}^{ek/t}(C)$ such that the product morphism $(f,g) : X \to C \times Z$ embeds $X$ as the degeneracy locus of a vector bundle map $\phi : \cO_{C \times Z}^{\oplus t} \to f^{\ast}\cM \otimes g^{\ast}\cG$, and the cokernel of $\phi^{\vee}$ is a vector bundle of rank $k$ on $X.$}
%\end{itemize}
\end{prop}
\begin{proof}
Let $p : C \times Z \to C$ and $q: C \times Z \to Z$ be the relevant projection maps; clearly $p|_{X}=f.$  We will show that $q^{\ast}\cF$ is an $f-$Ulrich sheaf on $C \times Z$.  By Lemma \ref{lma:gf-ulrich}, it is enough to show that $q^{\ast}\cF$ is a $(q^{\ast}\cG,p)-$Ulrich sheaf.  Grauert's theorem together with (i) implies that $R^{i}p_{\ast}q^{\ast}(\cF \otimes \cG^{\vee})=0$ for $i=0,1$; in addition, $p_{\ast}q^{\ast}\cF \cong H^{0}(\cF) \otimes \cO_{C}.$  This proves the first part of the statement.

Let $r \geq 1$ be given.  By \cite{DN}, there exist semistable vector bundles $\cF$ and $\cG$ on $Z$ of respective ranks $r$ and $s$ such that $H^{i}(\cF \otimes \cG^{\vee})=0$ for $i=0,1.$  Replacing $\cF$ and $\cG$ by twists if necessary, we can assume that $c_{1}(G) \geq 2s(h+1)$; this guarantees that $\cG$ is globally generated (e.g. \cite{Po}, 2.5).  In particular, the pullback $q^{\ast}\cG$ is globally generated.

Let $\cM$ be a globally generated line bundle on $C$ having positive degree.  We claim that there exists a morphism $\phi : \cO_{C \times Z}^{s} \to \cM \boxtimes \cG$ whose degeneracy locus $X$ is a smooth irreducible curve.  Since $\cM \boxtimes \cG$ is globally generated, Theorem II of \cite{FL} and Teorema 2.8 of \cite{Ott} imply that the claim is justified once $\cM \boxtimes \cG$ is shown to be an ample vector bundle.  For a given point $z \in Z,$ we have $\cM \boxtimes \cG \cong (\cM \boxtimes \cO_{Z}(z)) \otimes q^{\ast}(\cG(-z)).$  The bundle $\cG(-z)$ is semistable of degree $c_{1}(\cG(-z)) \geq s(2h+1) > 0$, so a theorem of Hartshorne (e.g. \cite{Laz} Theorem 6.4.15) implies that it is ample; consequently, $q^{\ast}\cG(-z)$ is semiample.  Since $\cM \boxtimes \cO_{Z}(z)$ is ample, it follows that $\cM \boxtimes \cG$ is ample.
\end{proof}

\begin{rmk}
Given a morphism of curves $f: X \to C$, one could use the existence of a finite morphism $g : X \to \mathbb{P}^{1}$ and attempt to construct an $f-$Ulrich sheaf by applying the previous considerations to the product morphism $(f,g) : X \to C \times \mathbb{P}^{1}.$  However, the singularities of $(f,g)$ may be very hard to control.  
\end{rmk}

\subsection{Ramification}
We recall the following Clifford-type theorem of Mercat for semistable vector bundles.  
\begin{thm}
\label{merc-cliff}
(\cite{Mer}, Theorem 2.1) Let $X$ be a smooth curve of genus $g_{X} \geq 5$ which is not hyperelliptic, trigonal, or a plane quintic, and let $E$ be a semi-stable vector bundle on C of rank $r,$ degree $d$ and slope $\mu=d/r.$
\begin{itemize}
\item[(i)]{If $2+\frac{2}{g_{X}-4} \leq \mu \leq 2g_{X}-4-\frac{2}{g_{X}-4},$ then $h^{0}(E) \leq \frac{d}{2}.$}
\item[(ii)]{If $1 \leq \mu \leq 2+\frac{2}{g_{X}-4},$ then $h^{0}(E) \leq \frac{1}{g_{X}-2}(d-r)+r.$ \hfill \qedsymbol}
\end{itemize}
\end{thm}

A classical result of Noether implies that the gonality of any smooth plane curve of degree $d \geq 4$ is equal to $d-1.$  The following consequence will be used in the sequel.

\begin{lma}
\label{lma:plane-curve-gonality}
Any smooth projective curve with gonality 5 or greater satisfies the hypotheses of Theorem \ref{merc-cliff}. \hfill \qedsymbol
\end{lma}  

Proposition D follows immediately the next result and (i) of Proposition \ref{prop:bielliptic-boundary}.

\begin{prop}
\label{bound-final-form}
Let $f:X \rightarrow Y$ be a finite flat morphism of degree $d \geq 3,$ where $X$ and $Y$ of smooth projective curves with respective genera $g_{X}$ and $g_{Y}.$  Assume that $g_{Y} \geq 1$ and that $X$ is neither hyperelliptic, trigonal, nor a plane quintic.  Then if $X$ admits an $f-$Ulrich bundle, we have that
\begin{equation}
\label{corollary-bound-2}
d \leq b/4.
\end{equation} 
\end{prop}

\begin{proof}
Assume $X$ admits an $f-$Ulrich bundle $E$ of rank $r \geq 1.$  Then $E$ is semistable with $h^{0}(E)=dr$ and $\mu=\frac{c_{1}(E)}{r}=\frac{b}{2}.$  To verify that $E$ satisfies the hypotheses of Theorem \ref{merc-cliff}, we now rule out the possibility that $\frac{b}{2} = \mu > 2g_{X}-4-\frac{2}{g_{X}-4}.$  If the latter holds, then since $g_{X} \geq 5$ we have 
\begin{equation}
-d(g_{Y}-1) >  g_{X}-3-\frac{2}{g_{X}-4} \geq 2-\frac{2}{g_{X}-4} \geq 0
\end{equation}
which cannot happen, given that $g_{Y} \geq 1.$  We will be done once we check that (i) of Theorem \ref{merc-cliff} applies, since the desired inequality $d \leq b/4$ will then follow at once.  If (ii) applies, then $h^{0}(E) \leq \frac{1}{g_{X}-2}(c_{1}(E)-r)+r.$  This is equivalent to $d \leq \frac{b-2}{2g_{X}-4}+1,$ which may in turn be rewritten as 
\begin{equation}
d \leq \frac{2}{1+\frac{g_{Y}-1}{g_{X}-2}} 
\end{equation}
The right-hand side is at most 2, so this is clearly impossible.
\end{proof} 

\begin{rmk}
There is more evidence for the principle that a finite flat morphism $f:X \to Y$ must be sufficiently ramified in order to admit an Ulrich sheaf.  If $X$ is reduced, connected and $f:X \to Y$ is unramified then $X$ does not admit an $f$-Ulrich sheaf unless $f$ is an isomorphism.  Indeed, if $\cE$ is an $f$-Ulrich sheaf then $f_*\cE \cong \cO_Y^{\oplus N}$ for some $N$, by definition.  We consider the associated $\cO_Y$-algebra homomorphism $f_*\cO_X \to M_n(\bk) \tk \cO_Y$.  Since $f$ is unramified $f_*\cO_X \cong (f_*\cO_X)^\vee$ as $\cO_Y$-modules.  Hence $\Hom(f_*\cO_X, \cO_Y) = \rmH^0(X,\cO_X)=\bk$ and we conclude that $f_*\cO_X \to M_n(\bk) \tk \cO_Y$ factors through the trace $f_* \cO_X \to \cO_Y$.  The only way this can happen is if the trace $f_*\cO_X \to \cO_Y$ is an isomorphism and hence $f$ is an isomorphism.  
\end{rmk}

We now turn to the sharpness of the bound in Proposition \ref{bound-final-form}.

\begin{prop}
\label{prop:bielliptic-boundary}
For each $d \geq 3,$ there exists a finite morphism $f_{d} : X_{d} \rightarrow Y_{d}$ of smooth projective curves satisfying the following properties:

\begin{itemize}
\item[(i)]{$f_d$ has degree $d,$ and its branch divisor has degree $4d.$}
\item[(ii)]{$X_{d}$ has a unique bielliptic structure $h_{d}:X_{d} \rightarrow E$ (where $E$ is an elliptic curve) and is neither hyperelliptic, trigonal, nor a smooth plane quintic.}
\item[(iii)]{$f_d$-Ulrich line bundles exist and are parametrized by $h_{d}^{\ast}({\rm Pic}^{d}(E)) \subset {\rm Pic}^{2d}(X_{d}).$}
\end{itemize}
\end{prop}
\begin{proof}
Let $d \geq 3$ be given, let $Y_d$ be a hyperelliptic curve of genus $d-1,$ and let $E$ be an elliptic curve.  If $|\mathcal{H}|$ is the hyperelliptic pencil on $Y_d$ and $\mathcal{M}$ is a fixed line bundle of degree $d$ on $E,$ the line bundle $\mathcal{H} \boxtimes \mathcal{M}$ on the surface $Y_d \times E$ is ample and globally generated with $2d$ global sections, so there is a smooth irreducible curve $X_d$ in the linear system $|\mathcal{H} \boxtimes \mathcal{M}|,$ which we fix for the rest of the proof.  

Let $f_d : X_d \rightarrow Y_d$ be the finite morphism of degree $d$ which is induced by projection onto the first factor.  The genus of $X_d$ is equal to $d^{2}+1$ by the adjunction formula, so Riemann-Hurwitz implies that the branch divisor of $f_d$ has degree $4d$; thus (i) is proved.

Given that $X_d$ is bielliptic by construction and of genus at least 10, the Castelnuovo-Severi inequality implies that the bielliptic structure $h_d : X_d \rightarrow E$ is unique and that $X_d$ is neither hyperelliptic, trigonal nor a plane quintic; thus (ii) is proved. 

We claim that for any $\mathcal{M}' \in {\rm Pic}^{d}(E),$ the line bundle $\mathcal{L}:=h_{d}^{\ast}\mathcal{M}'$ is an Ulrich line bundle for $f_d$.  Since $d \geq 3$ it follows that $\mathcal{M}'$ is globally generated, so the same is true of $\mathcal{L};$ we also have that $c_{1}(\mathcal{L})=2d.$  By Lemma \ref{ulrich-curves-criterion}, it is enough to check that $h^{0}(\mathcal{L})=d.$  If $B/2$ is a square root of the branch divisor of $h_{d},$ then $B/2$ has degree $d^{2},$ so
\begin{equation}
\rmH^{0}(\mathcal{L}) \cong \rmH^{0}(\mathcal{M}' \otimes h_{\ast}\mathcal{O}_{X}) \cong \rmH^{0}(\mathcal{M'}) \oplus \rmH^{0}(\mathcal{M}'(-B/2)) \cong \rmH^{0}(\mathcal{M}')
\end{equation}
and it follows that $h^{0}(\mathcal{L})=h^{0}(\mathcal{M}')=d.$  In order to prove (iii) it remains to verify that any $f_{d}-$Ulrich line bundle is of the form $h_{d}^{\ast}\mathcal{M}'$ for some $\mathcal{M}' \in {\rm Pic}^{d}(E).$

Let $\mathcal{L}$ be an $f_{d}-$Ulrich line bundle.  Then $\mathcal{L}$ is globally generated of degree $2d$ with $d$ global sections, and so determines a morphism $\phi_{\mathcal{L}}:X_{d} \rightarrow \mathbb{P}^{d-1}$ whose image is not contained in any hyperplane.  The proof will be concluded once we show that $\phi_{\mathcal{L}}$ must factor through $h_{d} : X \rightarrow E.$

We consider two separate cases.

\textbf{Case 1:} $d=3.$  $X_3$ has genus 10, and any degree-6 morphism from $X_3$ to $\mathbb{P}^{2}$ which is birational onto its image must be an embedding of $X_3$ as a plane sextic curve.  However, $X_3$ does not admit a degree-4 pencil, which contradicts the biellipticity of $X_{3}.$  The only possibility is that $\phi_{\mathcal{L}}$ factors through a double covering of a smooth plane cubic; the latter map must be $h_{3}$ by the uniqueness of the bielliptic structure on $X_{3}.$

\textbf{Case 2:} $d \geq 4.$  Castelnuovo's bound yields different inequalities depending on whether $d=4,$ $d=5,$ or $d \geq 6.$  However, a straightforward comparison of these implies that if $\phi_{\mathcal{L}}$ is birational onto its image for $d \geq 4$ we have $d^{2}+1 \leq 3d+3,$ which is impossible.  Therefore $\phi_{\mathcal{L}}$ must factor through a covering of a nondegenerate curve $X'_{d} \subseteq \mathbb{P}^{d-1}$ of degree at least 2 and at most $d.$  The nondegeneracy implies that $X'_{d}$ must be of degree $d-1$ or $d,$ but since $d-1$ cannot divide $2d,$ we conclude that the degree of $X'_{d}$ must be $d,$ so that the covering $X_{d} \rightarrow X'_{d}$ is of degree 2.  Another application of the Castelnuovo bound implies that $X'_{d}$ must be rational or elliptic, and since $X_d$ is not hyperelliptic, $X'_{d}$ must be elliptic.  Consequently $X'_{d}=E$ and this double covering must be $h_{d}$ by the same reasoning as the previous case.
 \end{proof}

\begin{rmk}
\label{rmk:high-degree}
If $f:X \to \mathbb{P}^{m}$ is a finite flat morphism (where $m \geq 2$) then Riemann-Hurwitz applied to a line in $\mathbb{P}^{m}$ shows that the branch divisor of $f$ is of degree at least $2(d-1)$; any aCM curve in $\mathbb{P}^{m}$ satisfying the hypotheses of Proposition \ref{prop:general-lift} is of degree 3 or greater, and will therefore intersect the branch divisor of $f$ in degree at least $6(d-1)$.
\end{rmk}

Proposition \ref{bound-final-form} can be leveraged to obtain a lower bound for the ramification of a finite flat morphism $f:X \rightarrow Y$ of varieties \textit{of any dimension} $m \geq 1$ which admits an $f-$Ulrich sheaf.   
\begin{cor}
\label{cor-bound-final-form}
Let $f:X \rightarrow Y$ be a finite flat morphism of degree $d \geq 3$ with branch divisor $B \subseteq Y,$ where $X$ and $Y$ are projective varieties of dimension $m \geq 2$ which are nonsingular in codimension 1.  Then if $X$ admits an $f-$Ulrich bundle, we have that
\begin{equation}
\label{corollary-bound-2}
4d \leq \inf_{Z}\{B \cdot Z\}
\end{equation} 
where the infimum is taken over all curves $Z \subseteq Y$ satisfying the following properties:
\begin{itemize}
\item[(i)]{$Z$ is smooth and irreducible of genus 1 or greater.}
\item[(ii)]{$Z$ meets $B$ transversally in the smooth locus of $Y$.}
\item[(iii)]{The curve $X_{Z}:=f^{-1}(Z)$ is neither hyperelliptic, trigonal, nor a smooth plane quintic.  \hfill \qedsymbol }
\end{itemize}
Moreover, there exist curves $Z \subseteq Y$ satisfying (i), (ii) and (iii), and \textnormal{(\ref{corollary-bound-2})} is sharp for all $d$.  
\end{cor}

\begin{proof}
If a curve $Z$ satisfying (i), (ii) and (iii) exists, then a straightforward restriction argument applied to Proposition \ref{bound-final-form} yields (\ref{corollary-bound-2}); we now verify the existence of such curves.  Choose an embedding of $Y$ in projective space of dimension greater than $m,$ and let $\pi : Y \rightarrow \mathbb{P}^{m}$ be a generic linear projection associated to this embedding.  In addition, choose a 2-dimensional projective space $\Lambda \subseteq \mathbb{P}^{m}$ and a smooth plane sextic curve $Z' \subseteq \Lambda$ which meets $\pi(B)$ transversally, and let $Z:={\pi}^{-1}(Z').$  Our transversality assumption implies that $Z$ is smooth and irreducible.  Moreover, since $Z'$ is of genus 10 and has gonality equal to 5, and both gonality and genus can only increase upon passing to finite covers, it follows from Lemma \ref{lma:plane-curve-gonality} that $Z$ satisfies (i),(ii), and (iii).    

To see that (\ref{corollary-bound-2}) is sharp, we take for each $d \geq 3$ the curve morphism $f_{d}: X_d \rightarrow Y_d$ granted by Proposition \ref{prop:bielliptic-boundary} and consider $f_d \times {\rm id} : X_d \times \mathbb{P}^{m-1} \rightarrow Y_d \times \mathbb{P}^{m-1};$ for a general $w \in \mathbb{P}^{m-1},$ the curve $Z = Y_{d} \times \{w\}$ satisfies (i),(ii) and (iii).   
\end{proof}

\end{document}